\documentclass{amsart}

\usepackage[pagebackref=false,colorlinks=true, pdfstartview=FitV, linkcolor=blue,citecolor=red, urlcolor=blue]{hyperref}

\usepackage{amsmath,amsthm,amssymb,verbatim,mathrsfs}%
\usepackage{color}
\usepackage{hyperref}
\usepackage{url}
\newcommand{\bburl}[1]{\textcolor{blue}{\url{#1}}}

\newtheorem{thm}{Theorem}[section]

\newtheorem{lem}[thm]{Lemma}
\newtheorem{prop}[thm]{Proposition}
\newtheorem{prob}[thm]{Problem}
\theoremstyle{definition}

\theoremstyle{definition}

\theoremstyle{remark}
\newtheorem{rem}[thm]{Remark}

\newcommand\be{\begin{equation}}
\newcommand\ee{\end{equation}}
\newcommand\bee{\begin{equation*}}
\newcommand\eee{\end{equation*}}
\newcommand\ben{\begin{enumerate}}
\newcommand\een{\end{enumerate}}

\def\Q{\ensuremath {{\bf Q}}}

\def\CC{\ensuremath {{\bf C}}}
\def\H{\ensuremath {{\mathcal H}}}
\def\I{\ensuremath {{\mathcal I}}}

\def\R{\ensuremath {{\bf R}}}

\def\A{\ensuremath {{\bf A}}}

\def\BB{\ensuremath {{\mathcal B}}}
\def\OO{\ensuremath {{\mathcal O}}}

\def\L{\ensuremath {\mathscr L}}
\def\Z{\ensuremath {{\bf Z}}}

\def\C{\ensuremath {\mathscr C}}

\def\T{\ensuremath {\mathcal T}}

\def\g{\ensuremath {\mathfrak g}}

\def\a{\ensuremath {\mathfrak a}}
\def\tr{\ensuremath {\mathrm{tr}}}
\def\GL{\ensuremath {\mathrm{GL}}}
\def\bs{\ensuremath {\backslash}}

\numberwithin{equation}{section}

\title{Refinements of the trace formula for GL(2)}
\author{Tian An Wong}
\email{\textcolor{blue}{\href{wongtianan@math.ubc.ca}{wongtianan@math.ubc.ca}}}
\address{The University of British Columbia, Vancouver, BC, Canada}

\keywords{Arthur-Selberg trace formula, Beyond Endoscopy, Poisson summation, basic function}
\subjclass[2010]{Primary 11F72, Secondary 11F68}
\date{\today}

\begin{document}

\begin{abstract}
We express the discrete noncuspidal terms in the spectral side of the trace formula for GL(2) in terms of orbital integrals, obtaining a geometric expansion for the cuspidal part of the trace formula. Assuming the Ramanujan conjecture for GL(2), this is equal to the tempered part of the trace formula, providing an alternate approach to the method of Frenkel-Langlands-Ng\^o in the context of Beyond Endoscopy. Using this, we establish a formula which in principle leads to an $r$-trace formula, and conclude with some remarks regarding the primitivization of the trace formula. 
\end{abstract}

\maketitle



\section{Introduction}

Let $G$ be a reductive group over a number field $F$, and $\A$ the adele ring of $F$. The trace formula of $G$ is an identity of distributions
\[
J_\text{spec}(f) = J_\text{geom}(f),\qquad f\in C^\infty_c(G(\A))
\]
where on the spectral side the main terms are characters of irreducible automorphic representations of $G$, 
\[
\sum_{\pi\in\Pi(G)} a^G(\pi)f_G(\pi),
\]
with $a^G(\pi)$ being the multiplicity of $\pi$ in the  $L^2$-discrete spectrum of $G$ and $f_G(\pi)=\tr(\pi(f))$ is the character $\pi$, while on the geometric side the main terms are orbital integrals 
\[
\sum_{\gamma\in\Gamma(G)}a^G(\gamma) f_G(\gamma) 
\]
where $a^G(\gamma)$ is a certain global geometric coefficient, and $f_G(\gamma)$ is the normalized orbital integral of $f$ at the conjugacy class $\gamma$,
\[
f_G(\gamma) = |D(\gamma)|^{1/2}\int_{G_\gamma(\A)\bs G(\A)} f(x^{-1}\gamma x)dx,
\]
where $D(\gamma)$ is the discriminant of $\gamma$ and $G_\gamma$ is the centralizer of $\gamma$ in $G$. If $G = \GL(1)$, this reduces to the usual Poisson summation formula.   In \cite{FLN} it was proposed that an additive Poisson summation formula could be applied to the regular elliptic terms in $J_\text{geom}(f)$, in order to remove the contribution of the nontempered spectrum to $J_\text{spec}(f)$. This was carried out in a special case for $\GL(2)$ by Altu\u{g} \cite{Alt1}, furthering the preliminary analysis carried out by Langlands \cite{BE}.

We present here a different approach to the problem, by applying instead a muliplticative Poisson summation formula to the discrete, noncuspidal terms on the spectral side of trace formula for $G=\GL(2)$. This provides then a geometric expansion for the cuspidal spectral terms, modulo the continuous terms which are generally simpler to address. It may be helpful to note that there is some precedent for this: in making the trace formula for $\GL(2)$ invariant, Langlands applied a Poisson summation formula to the noninvariant terms on the geometric side, and combining with the noninvariant terms on the spectral side obtained invariant distributions \cite[\S10]{BC}. As pointed out by Arthur \cite[\S22]{Art}, the singularities of orbital integrals pose difficulties in generalizing Langlands' method to higher rank, and Arthur thus established the invariant trace formula for general $G$ by instead applying the summation formula to the noninvariant spectral terms. Returning to our setting, the singularities of orbital integrals continue to present difficulties, even for SL$_2$ \cite{ST} and $\GL(n)$ \cite{GKM}, thus it appears more amenable to manipulate the spectral distributions instead. 

Consider then the discrete part of the trace formula, which takes the form 
\[
I_\text{disc}(f) = \sum_{M}|W^M_0||W^G_0|^{-1}\sum_{s\in W(M)_\text{reg}}|\det(s-1)_{\a^G_M}|^{-1}\tr(M_P(s,0))\I_P(0,f))
\]
which involves the discrete spectrum and singular points in the continuous spectrum. Here $M$ runs over Levi subgroups of $G$ containing a fixed minimal Levi $M_0\subset G$, $W^G_0$ denotes the Weyl group of $G$, $M_P$ is the intertwining operator and $\I_P$ is the induced representation associated to the parabolic $P$ of $M$ respectively (see for example \cite[(1.5)]{problems}).  The cuspidal, tempered spectrum appears only in the term corresponding to $M=G$, written as
\[
I_\text{cusp}(f)  = \sum_{\pi\in\Pi_\text{cusp}(G)}a^G(\pi)f_G(\pi),
\]
which one would like to isolate. For $G=\GL(2)$, we obtain in Theorem \ref{temp} a geometric expansion for $I_\text{cusp}(f)$, minus the contribution of the continuous terms in the trace formula, that is, the difference $J_\text{cont}(f) = J_\text{spec}(f) - I_\text{disc}(f)$. In principle, our technique works seamlessly for dealing with the terms associated to $M\neq G$, whereas for $M=G$ we have needed $f$ to be spherical, thus warranting further consideration.

In the framework presented in \cite{problems}, the next step is then to establish an $r$-trace formula $I^r_\text{cusp}(f)$, where the spectral side of $I_\text{cusp}(f)$ is weighted with coefficients $m_r(\pi)$ which vanish if and only if the corresponding automorphic $L$-function $L(s,\pi,r)$ has a pole at $s=1$ for a fixed representation $r$ of the $L$-group $^LG$. With this in mind, we establish in Theorem \ref{trace} a weighted trace formula using the basic functions $f^r_s$ which satisfy 
$$\tr(\pi(f^r_s)) = L^V(s,\pi,r)\tr(\pi(f_V))$$
for automorphic representations $\pi$ unramified outside the finite set of valuations $V$ of $F$, and here $L^V(s,\pi,r)$ is the incomplete $L$-function \cite{ngosum}. In particular, we obtain a geometric expression whose meromorphic continuation to $s=1$ would establish the existence of the desired distribution $I^r_\text{cusp}(f)$. For technical reasons, we have elected to use the noninvariant trace formula, whereas in general this analysis should be carried out for the stable trace formula $S_\text{cusp}(f)$, to produce an $r$-stable trace formula $S^r_\text{cusp}(f)$. 

Finally, in Section \ref{prim} we make some remarks on the putative primitive trace formula $P^G_\text{cusp}(f)$, which one hopes to establish with the help of the $r$-stable trace formula. We note that \cite{Mok} has worked out special cases of the $r$-stable trace formula, and further obtained a weak form of the primitization of the the trace formula for odd orthogonal groups, using known results about the relevant $L$-functions and endoscopic classification, bypassing analysis of the geometric side, which will ultimately be needed in the general case.



\section{A tempered trace formula }

We shall use the formulation of the noninvariant trace formula for $\GL(2)$ given in \cite[\S6-7]{GJ}, which we refer the reader to for details. Let $G=\mathrm{GL}(2)$ over $F$ a number field, and $\bar{G}=G/Z\simeq\mathrm{PGL}(2)$, where $Z(F)\simeq F^\times$ is the centre of $G$. Fix a central character $\omega$ of $Z$.  Given an $\omega$-equivariant test function $f\in C_c^\infty(G(\A),\omega)$, we express the discrete part of the trace formula $I_\text{disc}(f)$ as the sum of 
\[
I_\mathrm{cusp}(f) = \sum_{\pi\in\Pi_\mathrm{cusp}(G)}f_G(\pi)
\]
where $\Pi_\text{cusp}(G)$ is the set of cuspidal automorphic representations of $G$, and 
\begin{align}
&+ \frac{1}{\text{vol}([\bar{G}])}\sum_{\chi^2=\omega}\int_{\bar{G}(\A)}f(x)\bar\chi(\det(x))dx \label{one}\\
&+ \frac{1}{4}\sum_{\chi^2=\omega}\tr(M(\chi,\chi)\pi_{(\chi,\chi)}(f)) \label{res}
\end{align}
with $\text{vol}([\bar{G}])=\text{vol}(\bar{G}(F)\bs\bar{G}(\A))$, and $\chi$ runs over characters of the norm one ideles $(F^\times\bs \A^\times)^1$, or equivalently, one-dimensional representations of $\bar{G}(\A)$ which necessarily factor through the determinant.

The spectral sum is  equal to the sum of orbital integrals
\[
J_\text{geom}(f) = \sum_{M\in \L}\sum_{\gamma\in\Gamma(M)}a^M(\gamma)J_M(\gamma,f)
\]
where $\L$ is the set of Levi subgroups of $G$, $a^M$ is the global geometric coefficient, $\Gamma(M)$ is the set of conjugacy classes of elements in $M$, and $J_M(\gamma,f)$ is the weighted orbital integral
\[
|D(\gamma)|^{1/2}\int_{G_\gamma(\A)\bs G(\A)}f(x^{-1}\gamma x)v_M(x)dx
\]
for a certain weight function $v(x)$, and the subtracted contribution from the continuous spectrum
\[
J_\mathrm{cont}(f) =  \frac{1}{4\pi}\sum_{\chi}\int_{-\infty}^\infty\tr(M(\eta)M'(\eta)\pi_{\eta}(f))d\eta,
\]
where $\eta=(\mu,\nu)$ is a  pair of quasi-characters of $F^\times\bs\A^\times$ such that $\mu\nu=\omega$, and $\pi_\eta$ is the principle series representation on $G$ induced from $\eta$, which is the space of functions $\varphi$ satisfying
\[
\varphi(\begin{pmatrix}t_1 & x \\ 0 & t_2\end{pmatrix}g)= \chi(t_1t_2)|t_1t_2^{-1}|\varphi(g),
\]
and $M(\eta)$ is the intertwining operator associated to $\eta$ and $\tilde{\eta}$. We will be interested in the case $(\mu,\nu) = (\chi\alpha^\frac12,\chi\alpha^{-\frac12})$, where $\alpha(x) = |x|$, which correspond to the terms \eqref{one} and \eqref{res}. Also write 
\[
H(\begin{pmatrix} t_1 & 0 \\ 0& t_2\end{pmatrix})=|t_1t_2^{-1}|.
\]

What we would like is to isolate the cuspidal contribution on the spectral side. Assume for simplicity for the rest of this paper that $\omega$ is trivial. 
\begin{thm}
\label{temp}
Let $f$ be an element in the spherical Hecke algebra $\mathscr H(G,K)$, and define the geometric expansion
\[
\tilde{J}_\mathrm{geom}(f) := J_\mathrm{geom}(f)
-  \sum_{\gamma \in \bar{M}(F)}\Big(\frac{\mathrm{vol}(K)^2}{\mathrm{vol}([\bar{G}])}f(\gamma )  - \frac{1}{4}H(\gamma ) f_G(\gamma )\Big).
\]
Then 
\be
\label{tempt}
I_\mathrm{cusp}(f) = \tilde{J}_\mathrm{geom}(f)  - J_\mathrm{cont}(f).
\ee
\end{thm}

\begin{proof}
Since $f$ belongs to the spherical Hecke algebra $\mathscr H(G,K)$, 
by the Cartan decomposition and bi-$K$-invariance we have
\[
\int_{\bar{G}(\A)}f(x)\bar\chi(\det(x))dx = \text{vol}(K)^2\int_{\bar{M}(\A)}f(a)\bar\chi(\det(a))dx,
\]
where $\bar  M$ is the split Cartan in the Iwasawa decomposition 
\[
\bar{G}=N\bar MK = \begin{pmatrix}1&x\\0&1\end{pmatrix}\begin{pmatrix}t&0\\0&1\end{pmatrix}K
\]
which we identify with the minimal Levi of $G$. The one-dimensional characters factor through the determinant, and are trivial on $N$ and $K$. Notice that the one-dimensional automorphic characters of $\bar{M}(A)$ are exactly $\psi = \chi\circ\det$, so we can view the right hand side as the Fourier transform $\hat{f}(\psi)$.

More precisely, the one-dimensional automorphic characters of $G$ can be seen as the composition
\[
G(\A) \stackrel{\det}{\to} \A^\times \stackrel{\chi}{\to} \CC^\times.
\]
 Since $\chi^2=1$, this descends to a map\[
\bar{G}(\A) \stackrel{\det}{\to}  \A^\times/(\A^\times)^2\stackrel{\chi}{\to} \CC^\times
\] 
 and as a quasicharacter $\chi$ is trivial on $F^\times \R^\times_+$ and hence unitary. Recall that a character of $\A^\times\simeq (\A^\times)^1 \times \R^\times_+$ factors as a product of a unitary character and $|\cdot|^t$ for $t\in\R$.
The quadratic characters thus range over the characters of the quotient 
\[
D:= F^\times \bs (\A^\times)^1/( \A^\times)^2. 
\]
This is a compact abelian group. We therefore think of the one-dimensional characters $\chi\circ\det$ as running over the character group 
\[
\text{Hom}(\bar{G},{\bf G}_m)\simeq  \text{Hom}(\bar M,{\bf G}_m) \simeq D^*
\]
In particular, the character group is  abelian. Our Poisson summation will be concentrated on the locally compact abelian group $\bar M$, whence
\[
\sum_{\gamma\in \bar{M}(F)} f(\gamma) = \sum_{\psi\in D^*}\int_{\bar{M}(\A)}f(\gamma)\psi(\gamma)d\gamma.
\]
Thus \eqref{one} is equal to
\[
\frac{\text{vol}(K)^2}{\text{vol}([G])}\sum_{\gamma\in \bar{M}(F)}f(\gamma).
\]

On the other hand, the intertwining operator is the scalar factor 
\[
M(\chi,\chi) = \frac{L(0,\chi^2)}{L(1,\chi^2)\epsilon(0,\chi^2)}
\]
and since $\chi^2=1$, this is equal to 
\[
\lim_{s\to 1} \frac{\zeta(1-s)}{\zeta(1+s)}=-1,
\]
we can therefore write \eqref{res} as 
\be
\label{res2}
-\frac{1}{4}\sum_{\chi^2=1} \tr(\pi_{(\chi,\chi)}(f)).
\ee
We compute this trace as
\[
\int_K \int_{N(\A)}\int_{\bar{M}(\A)}f(k^{-1} an k) \chi(\det(a))H(a)  da \ dn \ dk.
\]
The integrals converging absolutely, we may interchange the order of integration and define
\be
\label{Phi}
\Phi(a) = H(a)\int_K \int_{N(\A)}f(k^{-1} an k)  dn \ dk,
\ee
thus we can view $\tr(\pi_{(\chi,\chi)}(f))$ as a Fourier transform on $\bar M$ as before. Note that by a change of variables we can view this as an orbital integral since
\[
G_{a}(F_v)\bs G(F_v) = M(F_v)\bs P(F_v)K_v \simeq N(F_v) K,
\]
where $P$ is the parabolic containing defined by $N$, in other words
\[
\tr(\pi_{(\chi,\chi)}(f)) = \hat{\Phi}(\psi) = \int_{\bar{M}(\A)}H(a)f_G(a)\psi(a)da.
\]
Then applying the Poisson summation formula again \eqref{res2} is equal to 
\[
-\frac{1}{4}\sum_{\chi^2=1} \hat{\Phi}(\psi) = -\frac{1}{4}\sum_{\gamma \in\bar{M}(F)} \Phi(\gamma)
\]
as desired.
\end{proof}

We note that the unweighted orbital integrals $f_G(\gamma)$ that appear in $\tilde{J}_\text{geom}$ attached to $\gamma\in\bar{M}(F)$ do not occur in the original geometric expansion $J_\text{geom}(f)$, as the corresponding orbital integrals in the latter sum are weighted. We shall refer to \eqref{tempt} as a tempered trace formula, for assuming the Ramanujan conjecture all the terms appearing on either side are in fact tempered distributions.


\section{An $r$-trace formula}

We recall that the intention is to establish a certain $r$-trace formula in the sense of \cite{problems}, in which the spectral terms are weighted by a factor $m_r(\pi)$  pertaining to the order of the automorphic $L$-function $L(s,\pi,r)$ at $s=1$ for tempered automorphic representations $\pi$.  Let $V$ be a finite set of valuations of $F$ containing the archimedean places, and the places over which $G$ is ramified. Define the weighted spectral sum
\be
\label{Tr}
I^r_\mathrm{cusp}(f) = \sum_{\pi\in\Pi_\mathrm{cusp}(G)}m_r(\pi) f_{V,G}(\pi).
\ee
In the case of $G=\GL(2)$, the Ramanujan conjecture provides the expected description of the cuspidal tempered spectrum. Allowing this, the expression will be well-defined. There are two proposed definitions for the coefficient $m_r(\pi)$, which is as the order or residue of the pole of $L(s,\pi,r)$ at $s=1$. We shall compare the two with a view towards establishing $I^r_\text{cusp}(f)$ unconditionally. 

\subsection{First definition}

In \cite{BE} the weight factor was taken to be the order at $s=1$ of the unramified $L$-function of the tempered automorphic representation $\pi$ that is unramified outside of $V$,
\[
L^V(s,\pi,r) = \prod_{v\not\in V}\det(1-r(c(\pi_v))q_v^{-s})^{-1},
\]
where $q_v$ is the cardinality of the residue field of $F_v$, and $c(\pi_v)$ is the Satake parameter, or Frobenius-Hecke conjugacy class of $\pi_v$ in $\hat{G}(\CC)$. We may express $m_r(\pi)$ then as the residue of the logarithmic derivative 
\[
m_r(\pi) =  \text{Res}_{s=1}\frac{{L^V}'}{L^V}(s,\pi,r),
\]
Expressing it instead as the Mellin transform 
\[
\int_{\sigma-i\infty}^{\sigma+i\infty} \frac{{L^V}'}{L^V}(s,\pi,r) X^s ds
\]
for $\text{Re}(s)=\sigma$ large enough, and assuming that the only possible pole of $L(s,\pi,r)$ in the half plane Re$(s)\ge1$ is at the point $s=1$, $m_r(\pi)$ can be expressed by a usual contour-shift and an application of a Tauberian theorem as
\[
\lim_{N\to\infty}\frac{1}{|V_N|}\sum_{v\in V_N}\log(q_v)f^r_{v,G}(\pi)
\]
where $V_N = \{v\not\in V: q_v<  N\}$. The limit exists if $\pi$ is tempered, so \eqref{Tr} is well-defined. 

In order to make use of Theorem \ref{temp}, we require that $f_V$ belong to the spherical Hecke algebra $\H(G_V,K_V)$, and for all $v\not\in V$ we require that $f_v=f_v^r$ to be the element in the unramified spherical Hecke algebra such that by the Satake isomorphism
\[
f^r_{v,G}(\pi) = \tr(r(c_v(\pi)))
\]
for any unramified, smooth, irreducible representation of $G(F_v)$, and $f = f_V\prod_{v\not\in V} f_v^r$. To obtain a genuine trace formula, we would like a geometric expansion for it.

Assuming that the sums converge absolutely, we may interchange the sums to get
\begin{align*}
I^r_\mathrm{cusp}(f) &
=  \lim_{N\to\infty}\frac{1}{|V_N|}\sum_{v\in V_N}\log(q_v)\sum_{\pi\in\Pi_\mathrm{cusp}(G)}f^r_{v,G}(\pi)f_{V,G}(\pi)\\
&=  \lim_{N\to\infty}\frac{1}{|V_N|}\sum_{v\in V_N}\log(q_v)I_\mathrm{cusp}(f_Vf^r_v).
\end{align*}
Now we can apply Theorem \ref{temp}. What one would like to show is that the {\em individual} terms  on the geometric side converge to obtain a genuine geometric expression. Define
\[
I^r_\mathrm{geom}(f) =  \lim_{N\to\infty}\frac{1}{|V_N|}\sum_{v\in V_N}\log(q_v)I_\mathrm{geom}(f). 
\]
For the geometric expansion $I^r_\mathrm{geom}(f)$ to be well-defined, it suffices to show that the averaged sum of weighted orbital integrals 
\be
\label{lim}
\lim_{N\to\infty}\frac{1}{|V_N|}\sum_{v\in V_N}\log(q_v)J_M(\gamma,f_Vf_v^r)
\ee
converges. In \cite{W} the convergence of such a limit was studied, conditional upon certain estimates on the functions $f_v^r$ as $v$ varies. This appears to be a difficult problem in general, as the existence of such a limit is essentially equivalent to the analytic continuation of $L^V(s,\pi,r)$ to the half-plane Re$(s)>1$, and that after admitting the two assumptions above. 

\subsection{Second definition}

As pointed out in \cite{sarnak}, the alternative definition for the weighting factor
\[
m_r(\pi) = \text{Res}_{s=1}L^V(s,\pi,r)
\]
leads one to study sums of integers rather than sums of primes, which are more amenable to analytic number theoretic techniques. This was used in succesfully in \cite{Alt3} for the standard $L$-function of $G=\GL(2)$. It was also studied in \cite{FLN}, which led to the notion of the basic function $f^r_{v,s}$ over a nonarchimedean local field $F_v$, given by the expansion
\[
f^r_{v,s} = \sum_{n=0}^\infty f^{r,n}_{v,s},
\]
where $f^{r,n}_{v,s} \in \H(G_v,K_v)$ is the element of the Hecke algebra of $G_v = G(F_v)$ whose Satake transform is the invariant regular function on $\hat{G}$ given by 
$$
g \mapsto \tr(g,\text{Sym}^n(r)).
$$
 In our case, the sum is locally finite but has noncompact support in general. The key feature of the basic function is that for $\pi$ unramified over $F_v$ one has 
\[
\tr(\pi(f^r_{v,s}))= L_v(s,\pi,r),
\]
and vanishes if $\pi$ is ramified. The basic function belongs to a certain $r$-Schwartz space $\mathscr S^r(G(F_v))$ generalizing the usual Harish-Chandra Schwartz space on $G(F_v)$ \cite{Ngo}. The key observation here is that we may work unconditionally with $f^r_{v,s}$ at least for $\mathrm{Re}(s)$ large enough, rather than having to make assumptions about the analyticity of $L^V(s,\pi,r)$ as in the first definition. 

\begin{thm}
\label{trace}
Let $f_V\in \H(G_V,K_V)$, and set $f^r_s = f_V\prod_{v\not\in V}f^r_{v,s}$. Then for $\mathrm{Re}(s)$ large enough, for $G=\GL(2)$ we have that 
\[
I_\mathrm{cusp}(f^r_s)  = \sum_{\pi\in\Pi_\mathrm{cusp}(G)} L^V(s,\pi,r) \tr(\pi(f_V)) 
\]
is equal to 
\[
\tilde{J}_\mathrm{geom}(f^r_s) - J_\mathrm{cont}(f^r_s). 
\]
\end{thm}
\begin{proof}
We first recall the space of functions constructed in \cite{FLM} extending the usual space of test functions $C_c^\infty(\bar{G}(\A))$. For any compact open subgroup $K$ of $G(\A_f)$ the space $\bar{G}(\A)/K$ is a differentiable manifold. Any element $X\in \mathcal U(\bar{\mathfrak g})$, the universal enveloping algebra of the Lie algebra $\bar{\mathfrak g}$ of $\bar{G}(\R)$ defines a left-invariant differentiable operator $f*X$ on $\bar{G}(\A)/K$. Let $\C(\bar{G},K)$  be the space of 
smooth, right-$K_0$-invariant functions on $\bar{G}(\A)$ which belong to $L^1(\bar{G}(\A))$ together with all their derivatives. it is a Fr\'echet space under the seminorms
\[
||f*X||, \qquad X\in \mathcal U(\bar{\mathfrak g}).
\]
Denote by $\C(\bar{G})$ the union of $\C(\bar{G},K)$ as $K$ varies over open compact subgroups of $\bar{G}(\A_f)$, and endow $\C(\bar{G})$ with the inductive limit topology.

The function $\prod_{v\not\in V}f^r_{v,s}$ belongs to $\C(\bar{G})$ if Re($s$) is large enough. For finite nonarchimedean valuations $v\in V$, let $f_v \in L^1(G(F_v))$ and let $f_\infty$ be a smooth function on $G(F_\infty)$ where $F_\infty = \prod_{v|\infty}F_v$ such that
\[
||f_\infty * X||_{L^1(G(F_\infty))} 
\]
is finite for all $ X\in \mathcal U(\bar{\mathfrak g})$. Then by \cite[Corollary 1]{FLM} we may input $f^r_s$ into the usual trace formula to obtain the absolutely converging spectral expansion
\[
\sum_{\pi} \tr(\pi(f^r_s ))= \sum_\pi L^V(s,\pi,r) \tr(\pi(f_V))
\]
and \cite[Corollary 7.2]{FL} guarantees the absolute convergence of the corresponding geometric side. Here the sum $\pi$ runs over irreducible automorphic representations in the discrete spectrum of $G$ unramified outside $V$.

In order to apply Theorem \ref{temp}, we simply notice that $\H(G_V,K_V)$ embeds easily in $\C(\bar{G})$, and the desired formula
\[
I_\mathrm{cusp}(f^r_s) = \tilde{J}_\mathrm{geom}(f^r_s) - J_\mathrm{cont}(f^r_s )
\]
follows.
\end{proof}

\begin{rem}
While we could have written the dual expansion for $I_\mathrm{cusp}(f)$ as a sum of invariant distributions as in \cite{BC}, the results of \cite{FLM} and \cite{FL} are valid only for the noninvariant trace formula, hence we have chosen to work with the latter. In any case, the noninvariant terms on either side of the trace formula are not important to our present case.
\end{rem}

The $r$-trace formula in this context would then be the distribution obtained from as the residue
\[
I^r_\mathrm{cusp}(f_V) = \text{Res}_{s=1} I_\mathrm{cusp}(f^r_s ) = \sum_{\pi\in\Pi_\mathrm{cusp}(G)} m_r(\pi) \tr(\pi(f_V)).
\]
Once again, such an expression would require knowledge of the analytic continuation of relevant $L$-functions, amounting to the Ramanujan conjecture on average for $G$. But the advantage over the first method is that we have obtained a first approximation in Theorem \ref{trace}, and we are led to examine the analytic continuation of the geometric expansion $\tilde{J}_\mathrm{geom}(f^r_s f_V)$, and thus to the orbital integrals 
\[
J_M(\gamma,f^r_s) = |D(\gamma)|^{1/2}\int_{G_\gamma(\A)\bs G(\A)}f^r_s(x^{-1}\gamma x)v(x) dx,
\]
whose meromorphic continuation to $s=1$ is roughly equivalent to the existence of the limit \eqref{lim}. For Re($s$) large enough and fixed $\gamma$, the sum defining the basic function converges absolutely. It follows then that
\[
J_M(\gamma,f^r_s) = \sum_{n=0}^\infty J_M(\gamma,f^r_{n,s}).
\]
We observe that the structure of these orbital integrals resemble closely Igusa's local zeta functions. It may be that a deeper study of the connections between the two will yield the needed meromorphic continuation of the orbital integrals under consideration. Such a possible relation has been considered long ago by Langlands \cite{igusa}. 

In any case, we have the following expectation from Theorem \ref{trace}. 
\begin{prob}
Show that the distribution $\tilde{J}_\mathrm{geom}(f^r_s)$ has meromorphic continuation to the half-plane $\mathrm{Re}(s)\ge1$, with the only possible pole at $s=1$. 
\end{prob}

\noindent This  can be verified in some special cases. For example, if $r$ is the standard representation, then $f^r_s$ is simply the characteristic function of Mat$_2(\mathcal O_v)$ twisted by $|\det|^s$. In this case the unweighted orbital integrals are covered by the results of \cite{CD}, suggesting that the motivic approach may be promising, though we caution that in Igusa theory the functions generally under consideration are compactly supported, which is not the case for general $f^r_s$. 

\begin{rem}
It is important to note that even if $m_r(\pi)=0$ for all cuspidal tempered $\pi$, the integrals $J_M(\gamma,f^r_s)$ need not be holomorphic at $s=1$, as shown by the results of \cite{Alt3}. Moreover, the difference between the original expansion ${J}_\mathrm{geom}(f^r_s)$ and modified expansion $\tilde{J}_\mathrm{geom}(f^r_s)$ will be critical, so that it will be important to consider not simply the individual integrals $J_M(\gamma,f^r_s)$ but the entire distribution $\tilde{J}_\mathrm{geom}(f^r_s)$. 
\end{rem}

\section{On primitivity}
\label{prim}
From the $r$-trace formula, one would like a certain primitization of the stable trace formula of the form 
\[
S^r_\text{cusp}(f) = \sum_{G'}m_{G'}(r)\iota(G,G')\hat{P}^{\tilde{G}'}_\text{cusp}(f') 
\]
analogous to the stablization of the invariant trace formula. We shall be content with making only a few observations here. One of the problems inherent in attaining such a decomposition is that the sum is not well-defined. In the case of stabilization, the sum is taken over elliptic endoscopic data of $G$, whereas in \cite{problems} the proposed sums are taken over so-called elliptic beyond endoscopic data $(G',\mathscr G',\xi')$. Here $G'$ is a quasisplit group, $\mathscr G'$ is a split extension of the global Weil group $W_F$ by the dual group $\hat{G}'$ with an $L$-embedding $\xi':\mathscr G'\to {^LG}$. Also, let $\tilde{G}'$ be a split central extension of $G'$ by an induced torus $\tilde{C}'$ over $F$, and $\tilde{\xi}':{\mathscr G}'\to {^L\tilde{G}'}$ an $L$-embedding. The additional coefficient $m_{G'}(r)$ is intended to be the dimension datum of $G'$ at $r$, the multiplicity of the trivial representation of $\mathscr G'$ in the composition $r\circ\xi'$. It is now known that these data together are not sufficient to rigidify the problem of isolating the functorial source or sources of a given representation $\pi$ of $G$ for which $L(s,\pi,r)$ has a pole at $s=1$ \cite{AYY}. 

The linear forms $P^G(f)$ are to be defined inductively as
\[
P^G_\text{cusp}(f) = S_\text{cusp}^G(f) - \sum_{G'\neq G}\iota(G,G')\hat{P}^{\tilde{G}'}_\text{cusp}(f'),
\]
and are primitive in the sense that $P^G_\text{cusp}(f)$ should be the spectral contribution of cuspidal tempered automorphic representations that are not functorial transfers from a smaller group.  In the absence of functoriality, we shall opt for an alternate, but what should be ultimately equivalent, definition, calling $\pi$ primitive if $L(s,\pi,r)$ does not have a pole at $s=1$ for any nontrivial representation $r$ of $^LG$. 

\subsection{} The definition of $P^G_\text{cusp}(f)$ involves the test function $f'$, which is the stable transfer of the test function $f$ from $G$ to $G'$. It is the image of $f$ under the mapping defined by
\[
f'(\phi') = f^G(\tilde{\xi}'\circ\phi')
\]
where $\phi'$ ranges over bounded $\tilde{\eta}_V'$-equivariant Langlands parameters of $\tilde{G}_V$, where $\tilde{\eta}'$ is a character of $\tilde{C}'(F)\bs \tilde{C}'(\A)$. Recall that in the case of stability, the Langlands-Shelstad transfer is given by
\[
f'_v(\delta_v') = \sum_{\gamma_v\in \Gamma_\text{reg}(G_v)}\Delta(\delta_v',\gamma_v) f_v(\gamma_v) 
\]
where $\Delta(\delta_v',\gamma_v)$ is the transfer factor and $f'_v(\delta_v')$ is a stable distribution, and $\delta_v$ is a strongly regular stable conjugacy class of $\tilde{G}'_v$. The key observation here is that the notion of $L$-indistinguishability is determined by stability \cite{LL}, which can be interpreted on both sides of the trace formula. Is there a geometric condition for primitivity that can be used to characterize the distributions in the desired geometric expansion for $P^G_\text{cusp}(f)$? Certainly a deeper study of the stable-stable transfer in \cite{ST} and \cite{shelstad} will be important. 

\subsection{} We mention here a possible first approximation. With regards to the spectral side, in \cite{CFM} it was shown Arthur packets can be obtained from perverse sheaves on the variety $V_\lambda\simeq \mathbb A^d$ of Langlands parameters with a prescribed infinitesimal character $\lambda$. The action of $Z_{\hat{G}}(\lambda)$ gives a stratification of $V_\lambda$ into a finite disjoint union
\[
V_\lambda = \coprod_i C_i
\]
with the unique open orbit $C_0$ corresponding to the tempered representations. On the other hand, on the geometric side, Arthur has constructed a stratification of the space of characteristic polynomials, parametrizing conjugacy classes on $\GL_{n+1}$, or more generally the Steinberg-Hitchin base for general $G$,
\[
\Xi(n)  = \coprod_{m+1|n+1}\Xi(m,n)
\]
where the open subset $\Xi(n,n)$ corresponds to the cuspidal tempered spectrum classified by Moeglin and Waldspurger \cite{artstrat}. Deeper study of this open stratum and its relation to the characterization of \cite{CFM} may lead to insight into the problem of primitivity. we note here that both stratifications exist for a general reductive group. 

For example, in the case of $G=\GL_4$, we have the parallel stratifications 
\[
V_\lambda = C_0 \sqcup C_1 \sqcup C_2
\]
and
\[
\Xi(3) = \Xi(3,3) \sqcup \Xi(1,3)\sqcup \Xi(0,3),
\]
where the open subset $C_0$ corresponds to the open stratum $\Xi(3,3)$.

\subsection{} Finally, we note that the existence of a pole at $s=1$ has been proved in certain cases through integral representations of $L(s,\pi,r)$ such as in \cite{GJS}, relating it to the nonvanishing of certain global periods, and locally is related to the nonvanishing of certain Hom spaces. This suggests that techniques related to the Gan-Gross-Prasad conjectures may be of use in this regard, and the connection to spherical varieties \cite{yiannis}. 

\subsubsection*{Acknowledgments} I thank Julia Gordon for encouragement and comments on a preliminary version of this paper, and Clifton Cunningham for enlightening discussions concerning his work.

\bibliography{BESL2}
\bibliographystyle{alpha}
\end{document}